\documentclass{amsart}

\usepackage{amsmath,amssymb,amsthm,enumerate}
\usepackage{eucal,mathrsfs}
\usepackage[bbgreekl]{mathbbol}
\usepackage{color}
\usepackage[all]{xy}
\usepackage{pxfonts}

\theoremstyle{plain}
\newtheorem{lem}{Lemma}[section]

\newtheorem{thm}[lem]{Theorem}
\newtheorem*{thm*}{Theorem}

\newtheorem{prop}[lem]{Proposition}

\newtheorem{conj}[lem]{Conjecture}
\theoremstyle{definition}
\newtheorem{defn}[lem]{Definition}
\newtheorem{rem}[lem]{Remark}

\newtheorem{notn}[lem]{Notation}

\newcommand{\mbb}[1]{\mathbb #1}

\newcommand{\mc}[1]{\mathcal #1}
\newcommand{\ms}[1]{\mathscr #1}

\newcommand{\oper}[1]{\operatorname{#1}}

\newcommand\et{{\acute{e}t}}

\usepackage[OT2,T1]{fontenc}
\DeclareSymbolFont{cyrletters}{OT2}{wncyr}{m}{n}
\DeclareMathSymbol{\Sha}{\mathalpha}{cyrletters}{"58}

\newcommand{\Br}{\oper{Br}}

\newcommand{\per}{\oper{per}}
\newcommand{\ind}{\oper{ind}}

\newcommand{\Gal}{\oper{Gal}}

\newcommand{\frc}{\oper{frac}}

\newcommand{\Spec}{\oper{Spec}}

\newcommand{\subnorm}{\triangleleft}

\newcommand{\ov}{\overline}
\newcommand{\til}{\widetilde}
\newcommand{\wh}{\widehat}

\newcommand{\fun}[1]{\underline{\mc Fun / #1}}
\newcommand{\abfun}[1]{\underline{\mc Ab / #1}}

\newcommand{\eind}{\oper{eind}}
\newcommand{\rind}{\oper{rind}}

\newcommand{\comm}[1]{\underline{\mc C\mc Alg / #1}}

\newcommand{\pres}[1]{\underline{\mc Pres}}
\newcommand{\f}[1]{\textbf{h}_{#1}}

\newcommand{\cd}{\oper{cd}}

\usepackage{tipa}
\usepackage{hieroglf}
\newcommand{\R}{\oper{R}}

\newcommand{\Pf}{\oper{Pf\,}}

\title[Period-index, symbol length, and generic splittings]{Period and
index, symbol lengths, and generic splittings in Galois cohomology}
\author{Daniel Krashen}

\thanks{This research was partially supported by NSF grants
DMS-1007462 and DMS-1151252}

\begin{document}

\begin{abstract}
We use constructions of versal cohomology classes based on a new
notion of ``presentable functors,'' to describe a relationship between
the problems of bounding symbol length in cohomology and of finding the
minimal degree of a splitting field. The constructions involved are
then also used to describe generic splitting varieties for degree $2$
cohomology with coefficients in a commutative algebraic group of
multiplicative type.
\end{abstract}

\maketitle

%\tableofcontents

\section{Introduction}
For a field $F$ containing an $\ell$'th root of unity, and a central
simple $F$-algebra $A$ of period $\ell$, it follows from the seminal
result of Merkurjev and Suslin \cite{MS:NRH}, that $A$ is similar to
a product of symbol algebras of degree $\ell$. The minimal number of
symbols needed in such a presentation is called the symbol length of
$A$. The question then becomes, for a given field $F$, what is the
minimal number of symbols which are necessary to express any central
simple $F$-algebra of period $\ell$, or in other words, what is the
maximum symbol length of such an algebra? One may think of this
measure as an attempt to bound the potential ``complexity'' of the
structure of central simple $F$-algebras.

This question may naturally be generalized to higher degree cohomology
classes as well. Thanks to the Bloch-Kato conjecture/norm residue
isomorphism Theorem, proved in \cite{Voe:modl} (see also
\cite{Wei:NRI}), no longer assuming that $\mu_\ell \subset F$, if we
are given a cohomology class $\alpha \in H^n(F, \mu_\ell^{\otimes
n})$, we know that $\alpha$ may be written as a sum of some number of
symbols, where a symbol in this case refers to an $n$-fold cup product
of classes in $H^1(F, \mu_\ell)$. We may then ask, how many symbols
are necessary to express our given class $\alpha$ -- this is the
symbol length of $\alpha$. As before, we ask: ``what the maximum value
of the length over all possible classes in $H^n(F, \mu_\ell^{\otimes
n})$?''

One may also ask a similar question for quadratic forms. Namely, we
know that the powers of the fundamental ideal $I^n(F)$ are generated
by $n$-fold Pfister forms as Abelian groups. We can then ask how many
$n$-fold Pfister forms are necessary to write a given element of
$I^n(F)$? 

Conjecturally, these questions should be closely connected with other
measurements of complexity of the arithmetic of the field $F$, such as
the cohomological dimension, the Diophantine dimension, the
$u$-invariant of $F$, as well as its Brauer dimension (defined as the
minimum $d$ such that $\ind \alpha | (\per \alpha)^d$ for all Brauer
classes $\alpha \in \Br(F)$). The question of determining the Brauer
dimension is a special case of the period-index problem in cohomology. 

\smallskip

In this paper, we draw connections between these various measures of
complexity and describe new relationships between them. In part, our
method is to construct ``generic cohomology classes,'' of various
types using objects which we call ``presentable functors,'' a concept
somewhat related to O'Neil's notion of ``sampling spaces''
\cite{ONeil:SS,ONeil:MG1C}. In addition, the new technique of
presentable functors allows for the writing down of new generic
splitting varieties for cohomology classes, which we do at the end of
the paper.

\section*{Acknowledgments}
The material in this paper came as a direct result of discussions and
questions which arose at the workshop ``Deformation theory, patching,
quadratic forms, and the Brauer group,'' organized by the author and Max
Lieblich, held at the American Institute of Mathematics in 2011, and the
author benefited from a number of conversations from the participants of
that conference. The author would like to thank Parimala and A.  Merkurjev
for useful conversations during the writing of this paper, and Suresh and
Parimala for pointing out inconsistencies in a preliminary version of the
manuscript. In addition, the author would like to thank the anonymous
referee for a large number of helpful comments and corrections.

The work presented is closely related to the work of Saltman
\cite{Sal:CSL}, which also had its origins in the same workshop,
although these works were done independently and using different
methods.  The paper \cite{Sal:CSL}, written prior to this one,
considers the problem of symbol length in mod-$2$ cohomology based on
a finite $u$-invariant assumption and in this context give not only
the boundedness results obtained in this paper, but also more detailed
information concerning solvable lengths of Galois groups which split
cohomology classes. The present paper provides results concerning
cohomology with more general coefficient groups, and also addresses
related topics such as Pfister length, the period-index problem, and
generic splitting varieties.

\section{Overview and general remarks}

In this section we briefly give an overview of the results of the
paper.

\subsection*{Symbol length}

Suppose for a field $F$, we may bound the degrees of splitting fields
for cohomology classes up to a certain degree $n$. In this case, we we
may bound the symbol length of cohomology classes in degree $n$. More
precisely:

\begin{defn} \label{effective index definition}
For $\alpha \in H^i(F, \mu_\ell^{\otimes i})$, we define the
\textit{effective index} of $\alpha$, denoted $\eind(\alpha)$ to be
the minimum degree of a finite separable field extension $E/F$ such
that $\alpha_E = 0$.
\end{defn}

\begin{defn} \label{length definition}
For $\alpha \in H^i(F, \mu_\ell^{\otimes i})$, we define the length of
$\alpha$, denoted $\lambda(\alpha)$ to be the smallest $m$ such that
$\alpha$ may be written as a sum of $m$ symbols.
\end{defn}

\begin{thm*}[\ref{main bound}]
Let $F$ be a field and fix a positive integer $\ell$ prime to the
characteristic of $F$. Fix $d > 0$. Suppose that for every $m > 0$, there
is some $N_m$ such that for every cohomology class $\beta \in H^{d'}(L,
\mu_\ell^{\otimes d'})$ with $[L:F] \leq m$ and $d' < d$ we have $\eind
\beta \leq N_m$.  Then for $\alpha \in H^d(F, \mu_\ell^{\otimes d})$, we
may bound $\lambda(\alpha)$ in terms of the effective index of $\alpha$.
\end{thm*}

We note that in particular, the bound doesn't depend on the particular
field $F$ except in the determination of the constant $N$. In
section~\ref{coh2}, we present some evidence that $N$ might
be bounded in terms of the Diophantine dimension of the field $F$
and its finite extensions, and in fact prove that this is the case for
$\ell = 2$.

\begin{rem} \label{length to index}
We note that it is immediate that if $\alpha \in H^d(F, \ell^{\otimes
d})$ then we may bound the effective index of $\alpha$ in terms of
its symbol length.  To see this, suppose that we may write $\alpha =
\sum_{j = 1}^m \ov a^j$ with $\ov a^j = (a_1^j, \ldots, a_i^j)$.
One easily finds that the field \(L = F\left(\sqrt[\ell]{a_1^1},
\sqrt[\ell]{a_1^2}, \ldots, \sqrt[\ell]{a_1^m}\right)\) splits
$\alpha$.
\end{rem}

\subsection*{Generic splittings}

\begin{defn}
Let $A$ be a commutative group scheme and $\alpha \in H^d(F, A)$. We
say that a scheme $X/F$ is a generic splitting variety for $\alpha$
if for every field extension $L/F$ we have $X(L) \neq \emptyset$ if
and only if $\alpha_L = 0$.
\end{defn}

We note that for elements $\alpha$ of $H^d(F,
\mu_\ell^{\otimes^{d-1}})$ which are symbols (i.e. cup products of a
class in $H^1(F, \mbb Z/\ell)$ and $d-1$ classes in $H^1(F,
\mu_\ell)$), Rost's norm varieties, as constructed in
\cite{SJ:NV,HW:NVCL} are known to be generic splitting varieties for
$\alpha$ under the assumption that $\ell$ is a prime and that $F$ has
no prime-to-$\ell$ extensions. Specific constructions for symbols give
the existence of generic splitting varieties for symbols in small
degree, such as the Merkurjev-Suslin variety for a symbol in $H^3(F,
\mu_\ell^{\otimes 2})$. For non-symbols in degree at least $2$, the
Severi-Brauer variety gives the only known example of a generic
splitting variety for a general class $\alpha \in H^2(F, \mu_\ell) =
\Br(F)_\ell$. 

We show that one may construct such generic splitting
varieties for a general class in $H^2(F, A)$ for any 
commutative group scheme $A$ of multiplicative type: 
\begin{thm*}[\ref{generic splittings degree 2}]
Let $A$ be any commutative group scheme over $F$ of multiplicative
type and suppose $\alpha \in H^2(F, A)$.  Then there exists a smooth
generic splitting variety $X/F$ for $\alpha$.
\end{thm*}

\subsection*{Pfister numbers}

Recall that the fundamental ideal $I(F)$ of the Witt ring $W(F)$ of
equivalence classes of quadratic forms over $F$ is defined as the
ideal generated by even dimensional forms. One may show that the
$m$'th power of this ideal $I^m(F)$, is generated as an Abelian group
by the set of $m$-fold Pfister forms: that is, those forms which may
be written as an $m$-fold tensor product of those of the form
$\left<1, -a\right>$. For such a form $q$ with $[q] \in I^m(F)$, the
$m$-Pfister number of $q$, denoted $\Pf_m(q)$, is defined to be the
minimum number $n$ such that $q$ is equivalent to a sum of $n$
$m$-fold Pfister forms.

\begin{thm*}[\ref{pfister}]
Suppose that $F$ is a field of characteristic not $2$ in which $-1 \in
(F^*)^2$. Then the following are equivalent:
\begin{enumerate}
\item $u(F) < \infty$.
\item For all $m > 0$, there exists $N_m > 0$ such that
for all $n > 0$ and for all $\alpha \in H^n(L, \mu_2)$ with $[L:F] \leq m$,
we have $\eind \alpha < N_m$.
\item There exists $N > 0$ such that for all $n > 0$
and for all $\alpha \in H^n(F, \mu_2)$, we have $\lambda(\alpha) < N$.
\item There exists $N > 0$ such that for all $n > 0$
and for all $q \in I^n(F)$, we have $\Pf_n(q) < N$.
\end{enumerate}
\end{thm*}

\subsection*{Cohomological dimension}

It turns out that bounds on the symbol length and related notions in
Galois cohomology has striking consequences for cohomological
dimension. This can be seen using the reduced power operations, as
described in \cite{Kahn:FI,Revoy}. The result below is stated in
\cite[Proposition~1(8)]{Kahn:FI} in the case $\ell = 2$, but is easily
extended to general $\ell$ as well.

\begin{prop} \label{cdim bound}
Suppose that $F$ is a field such that the length of all classes in
$H^{2n}(F, \mu_\ell^{\otimes {2n}})$ is bounded by $m$. Then $H^{2n(m+1)}(F,
\mu_\ell^{\otimes {2n}(m+1)}) = 0$ if $\ell$ is odd or $-1 \in (F^*)^2$.
\end{prop}
\begin{rem}
In particular, it follows from this result combined with
Theorem~\ref{main bound} that a bound for the index of
Brauer classes of period $\ell$ gives a bound on the
$\ell$-cohomological dimension of $F$.
\end{rem}
\begin{proof}
Write a symbol $\alpha \in H^{2n(m+1)}\left(F, \mu_\ell^{\otimes
2n(m+1)}\right)$ as
$\alpha = \alpha_1 \cup \cdots \cup \alpha_{m+1}$ where $\alpha_i \in
H^{2n}(F, \mu_\ell^{\otimes {2n}})$.
We note that we have
\[\alpha = (\alpha_1 + \cdots + \alpha_{m+1})^{[m+1]}\]
where $[m]$ is the reduced power operation. But since $\sum_i
\alpha_i$ can also be written as a sum of at most $m$ symbols, it
follows that $\alpha = 0$.
\end{proof}

\section{Compactness of functors and bounding symbol length}
\label{functors}

\subsection{Functors}

Let $T$ be a commutative ring. By a \textit{$T$-algebra}, we mean a
ring $R$ together with a ring homomorphism $T \to R$. By convention,
all rings and ring homomorphisms are assumed to be unital.

\begin{defn}
A \textit{$T$-functor} is a functor from the category of commutative
$T$-algebras to the category of sets. We denote the category of
$T$-functors by $\fun T$.
\end{defn}

\begin{defn}
An \textit{Abelian $T$-functor} is a functor from the category of
commutative $T$-algebras to the category of Abelian groups. We denote the category of
Abelian $T$-functors by $\abfun T$.
\end{defn}

\begin{rem} We will generally abuse notation and consider the
(faithful) forgetful functor $\abfun T \to \fun T$ as an inclusion.
\end{rem}

\begin{notn}
Let $X$ be a $T$-scheme. We let $\f X$ denote the functor represented
by $X$. Similarly, for a commutative $T$-algebra $R$, we will abbreviate $\f
{\Spec(R)}$ by $\f R$.

\smallskip

\noindent
We will often abuse notation and write $X$ in place of $\f X$.
\end{notn}

\begin{defn}
We say that a morphism $\ms F : \ms X \to \ms Y$ for $\ms X, \ms Y \in
\fun T$ is \textit{surjective} (resp. \textit{injective}) if for every
commutative $T$-algebra $R$, the set map $\ms F(R) : \ms X(R) \to \ms Y(R)$ is
surjective (resp. injective). Similarly, we say that a sequence of
Abelian $T$-functors is exact if it yields an exact sequence when
evaluated at every commutative $T$-algebra $R$.
\end{defn}

\begin{defn}
Suppose that $\phi: T \to R$ is a ring homomorphism, and $\ms H \in \fun
T$. We let $\ms H_R \in \fun R$ be defined by $\ms H_R(S) = \ms H(S)$ where
a $R$-algebra $S$ is considered as a commutative $T$-algebra by composition
of the structure map of $S$ with $\phi$.
\end{defn}

\subsection{Presentations}

In this section we will develop the concept of ``presentable
functors.'' In fact, we will not so much be interested in this
particular property as much as a weaker condition for a functor to
satisfy, ``pointwise coverability,'' which we describe in
Section~\ref{coverings}. The usefulness of presentability is that the set
of presentable functors forms an Abelian category (see \ref{presentable
abelian}), and is therefore closed under useful constructions such as sums,
kernels and cokernels. These provide us with tools therefore to show that
functors are pointwise coverable.

\begin{defn}
Let $\ms H \in \abfun T$. We say that $\ms H$ is \textit{presentable}
if there exist affine $T$-group schemes $A_1$ and $A_2$ and an exact
sequence
\(A_1 \to A_0 \to \ms H \to 0.\)
In this case we say that $A_\bullet$ is a \textit{presentation} of
$\ms H$.
\end{defn}
\begin{rem}
Here we are abusing notation and writing $A_i, A_i'$ both for the
linear algebraic groups and the Abelian $T$-functors which they
represent.
\end{rem}
Note that if $\ms H, \ms H' \in \abfun T$ are both presentable, then
so is $\ms H \times \ms H'$ by taking a term by term product of their
two presentations. We write $\pres T$ to be the additive subcategory
of $\fun T$ consisting of presentable functors.

\begin{lem} \label{presentable abelian}
$\pres T$ is an Abelian subcategory of $\abfun T$.
\end{lem}
\begin{proof}
We need to show that $\pres T$ is closed under taking kernels and
cokernels in $\abfun T$.  Suppose that $\phi : \ms H \to \ms H'$ is a
morphism between presentable functors, and let $\ms K$ and $\ms C$ be
the kernel and cokernel respectively, computed in the category $\abfun
T$. Suppose that we are given
presentations
\(A_1 \to A_0 \to \ms H\) and \(A'_1 \to A'_0 \to \ms H'\). We may
extend $\phi$ to a commutative diagram of maps 
\[\xymatrix{
A_1 \ar[r] \ar[d]_{\phi_1} & A_0 \ar[r] \ar[d]_{\phi_0} & \ms H \ar[r]
\ar[d]_{\phi} & 0 \\
A'_1 \ar[r] & A'_0 \ar[r] & \ms H' \ar[r] & 0
}\]
To see that we can do this, consider the element of $a \in \ms H'(A_0)$
corresponding to the composition $A_0 \to \ms H \to \ms H'$. Since the
map $A_0'(A_0) \to \ms H'(A_0)$ is surjective, it follows that we may
find an element of $A_0'(A_0)$ mapping to $a$. But this element is
exactly the data of a morphism $\phi_0 : A_0 \to A_0'$ making the
right-hand square commute. Now, given the existence of $\phi_0$, we
consider the element $b \in A_0'(A_1)$ corresponding to the
composition $A_1 \to A_0 \overset{\phi_0}\longrightarrow A_0'$. It follows from the
commutativity of the right-hand square and the exactness of the top row
that $b$ maps to $0$ in $\ms H(A_1)$. By exactness of the bottom row,
it follows that there is an element of $A_1'(A_1)$ mapping to $b$ in
$A_0'(A_1)$. But, this element is exactly the data of a morphism
$\phi_1: A_1 \to A_1'$ making the diagram commute as desired.

We now turn to constructing presentations for $\ms K$ and $\ms C$. 
Consider the commutative diagram of morphisms of functors:
\[\xymatrix@R=.7cm{
A_1 \times A_1' \ar[r] \ar[d] & A_0 \times_{A_0'} A_1' \ar[r] \ar[d] & \ms K \ar[r]
\ar[d] & 0 \\
A_1 \ar[r] \ar[d]_{\phi_1} & A_0 \ar[r] \ar[d]_{\phi_0} & \ms H \ar[r]
\ar[d]_{\phi} & 0 \\
A'_1 \ar[r] \ar[d] & A'_0 \ar[r] \ar[d] & \ms H' \ar[r] \ar[d] & 0\\
A'_1 \times A_0 \ar[r] & A'_0 \ar[r] & \ms C \ar[r] & 0 \\
}\]
It follows from a diagram chase that the top row is a presentation for
$\ms K$ and the bottom row is a presentation for $\ms C$.
\end{proof}

\subsection{Coverings and compactness} \label{coverings}

\begin{defn}
A morphism $\phi : \ms H' \to \ms H$ with $\ms H', \ms H
\in \fun T$ is called a \textit{cover} if for every $T' \in \comm T$,
$\ms H'(T') \to \ms H(T')$ is surjective. It is called a
\textit{pointwise cover} if for every field $L$ with $T$-algebra
structure $T \to L$, the map $\ms H'(L) \to \ms H(L)$ is surjective.
\end{defn}

\begin{defn}
We say that $\ms H$ is \textit{coverable} (respectively
\textit{pointwise coverable}) if there exists a cover (respectively
pointwise cover) $\ms X \to \ms H$ with $\ms X$ represented by a
$T$-scheme of finite type.
\end{defn}

\begin{rem}
We note that this idea of pointwise cover is close to the notion of
the sampling space for a functor introduced by O'Neil \cite{ONeil:SS,ONeil:MG1C}. In our definition of pointwise coverability however, we
consider arbitrary field extensions as opposed to only finite
extensions as is done with sampling spaces.
\end{rem}

\begin{defn}
Suppose we have $T$-functors $\ms H', \ms H'' \subset \ms H$. We say that
$\ms H'$ is \textit{pointwise contained in} $\ms H''$ if for every
field $L$ with $T$-algebra structure $T \to L$ we have $\ms H'(L)
\subset \ms H''(L)$ as subsets of $\ms H(L)$.
\end{defn}

\begin{defn}
Let $\ms H \in \fun T$ and $\ms H' \subset \ms H$ a subfunctor. We
say that $\ms H'$ is \textit{pointwise quasi-compact in $\ms H$} if
for every collection of coverable subfunctors $\ms R_i \subset
\ms H$, $i \in I$ such that $\ms H'$ is pointwise contained in
$\bigcup_{i \in I} \ms R_i$, there exists a finite subset
$I' \subset I$ such that $\ms H'$ is pointwise contained in 
$\bigcup_{i \in I'} \ms R_i$.
\end{defn}

\begin{defn}
Let $\ms H \in \fun T$. We say that $\ms H$ is \textit{localizable} if 
given a domain $R \in \comm T$ with fraction field $L$, and $a, b \in \ms
H(R)$ with $a|_L = b|_L$, there exists $s \in R$ with $a|_{R[s^{-1}]}
= b|_{R[s^{-1}]}$.
\end{defn}

\begin{thm} \label{functor compactness}
Let $T$ be a Noetherian ring, let $\ms X \in \fun T$ be pointwise
coverable, and let $\ms H \in \fun T$ be localizable.  Suppose that
$\psi: \ms X \to \ms H$ is a morphism. Then $\psi(\ms X)$ is a
pointwise quasi-compact subfunctor of $\ms H$.
\end{thm}
\begin{proof}
Let $\phi_i: \ms U_i \to \ms H$, $i \in I$ be a collection of morphisms from
coverable functors such that $\psi(\ms X)$ is pointwise contained in
$\bigcup_{i \in I} \phi_i(\ms U_i)$. We wish to show that we may find
a finite subset $I' \subset I$ with $\psi(\ms X)$ also pointwise
contained in $\bigcup_{i \in I'} \phi_i(\ms U_i)$. 

Since each $\ms U_i$ is coverable, we may replace each by an affine
$T$-scheme $U_i = \Spec(S_i)$ (a cover), and preserve the hypothesis.

Choose a morphism from a finite type affine scheme $X \to \ms X$ giving a
pointwise cover of $\ms X$. Since $X$ is finite type, it has a finite
number of components $X_i$, and it suffices to check that the image of
each $X_i$ in $\ms H$ is pointwise contained in a union of a finite
number of subfunctors $\phi_i(U_i)$ of $\ms H$. In particular, we may
consider each $X_i$ individually, and hence reduce to the case where
$X$ is irreducible. Further, since our hypothesis only involved the
field points of $X$, we may also assume that $X$ is reduced by
replacing it with its reduced subscheme. In particular, we may reduce
to the case where $\ms X = X = \Spec(R)$ for an integral domain $R$.

We therefore replace $\ms X$ by an integral affine $T$-scheme $X =
\Spec(R)$, and proceed by induction on the Krull dimension of $X$.

In the case that $\dim(X) = 0$, we have $R$ is a field. The hypothesis
on surjectivity on $R$ points gives us for some $i$, a commutative diagram
\[\xymatrix@R=.3cm @C=2cm{
\Spec(R) \ar[r] \ar@{=}[rd] & U_i \ar[rd] \\
& X \ar[r] & \ms H
}\]
which gives a morphism $X \to U_i$ making the diagram commute. It follows
that for every commutative $T$-algebra $T'$ (and in particular for $T'$ a
field), we have $X(T') \subset U_i(T')$. We therefore obtain our result in
this case by choosing $I' = \{i\}$.

In the general case, assume that we have shown that the conclusion
holds for affine integral schemes $X$ when $\dim(X) < n$, and suppose
we are given an affine integral scheme of dimension $n$. Let $\eta \in
X$ be the generic point, with residue field $\kappa(\eta) = \frc(R)$.
By the hypothesis, we may find an index $i \in I$ and a point $q \in
U_i(\frc(R))$ such that we have a commutative diagram:
\[\xymatrix@R=.3cm @C=2cm{
\Spec(\frc(R)) \ar[r]^q \ar@{=}[rd] & U_i \ar[rrd]^{\phi_i} \\
& \eta \, \ar@{^(->}[r] & X \ar[r]_{\psi} & \ms H
}\]
This gives a morphism $\eta \to U_i$ and hence a rational map from $X$
to $U_i$. Choose an open set $V = \Spec(R[s^{-1}]) \subset X$ on which
this map is defined. We then have two elements of $\ms H(R[s^{-1}])$,
given by $V \to X \overset{\psi}\to \ms H$ and $V \to U_i
\overset{\phi_i}\to \ms H$. By the commutativity of the original
diagram, these points coincide when restricted to $\frc(R)$. Since
$\ms H$ is localizable, it follows that we may perform an additional
localization of $R[s^{-1}]$ so that these two elements agree.
Therefore, by adjusting our choice of $s$, we may assume that we have
a map $V \to U_i$, commuting with the morphisms to $\ms H$. Now, let
$Z = X \setminus V$. It follows that $\psi(V)$ is pointwise contained
in $\phi_i(U_i)$.

Since $\dim Z < \dim X$, it follows by induction
that we may find a finite subset $I'' \subset I$ such that $\psi(Z)$
is pointwise contained in $\bigcup_{j \in I''} \phi_j(U_j)$. But
therefore, since one easily sees that we have a pointwise containment
of $\psi(X)$ into $\psi(Z) \cup \psi(V)$, that $\psi(X)$ is pointwise
contained in $\bigcup_{j \in I'' \cup \{i\}} \phi_j(U_j)$, and hence
setting $I' = I'' \cup \{i\}$, we are done.
\end{proof}

\subsection{Relative Galois cohomology functors}

Let $G$ be a finite group, $T'/T$ a $G$-Galois extension of
commutative rings (as in \cite{Sal:GG,DeIn}), and $X \in \fun T'$.
Consider the Weil restriction $\R_{T'/T} X$. By
definition, this is the object in $\fun T$
defined by\footnote{This discussion make sense for more general extensions
$T'/T$, though we will not need this.}
\[\R_{T'/T}(R) = X(T' \otimes_T R).\]
The same definition applies to Abelian functors as well, associating to an
object $A \in \abfun T'$, an object $\R_{T'/T} A \in \abfun T$.

One may check that the Weil restriction of a presentable (Abelian) functor is
presentable (by taking Weil restrictions of its presentation). It 
It also follows from the definition, that the functor $\R_{T'/T} A$
admits an action by the Galois group $G$ induced by the action on $T'
\otimes_T R$.

\begin{defn}
Let $^G \ms H^i_{T'/T, A} \in \abfun T$ denote the functor defined by
$^G \ms H^i_{T'/T}(R) = H^i\left(G, A(T' \otimes_T R)\right)$.
\end{defn}

\begin{lem} \label{presentable cohomology}
Let $A \in \pres T$ be a presentable functor. Then the functor $^G \ms
H^i_{T'/T, A}$ is also presentable.
\end{lem}
\begin{proof}
Consider the presentable functor $C^j = \left(\R_{T'/T} A \right)^{|G|^j}$.
We think of elements of this group as functions from $G^j$ to $\R_{T'/T} A$
-- i.e. for a commutative $T$-algebra $R$, we have a natural bijection
\[ C^j(R) \leftrightarrow Map(G^j, \R_{T'/T}A (R)) = Map(G^j, A(T'
\otimes_T R)),\]
which we consider as an identification. Note that these Abelian groups
carry a natural action of the Galois group $G = Gal(T'/T)$. We may
obtain morphisms of group schemes $\delta: C^j \to C^{j+1}$ by the
standard formula:
\[(\delta \phi) (g_0, \ldots, g_{j+1}) = g_0(\phi(g_1, \ldots, g_j))
+ \sum_{i = 1}^n (-1)^i \phi(g_0, \ldots, \wh{g_i}, \ldots, g_j).\]
We thus obtain, for every choice of $j+1$-tuple $(g_0, \ldots, g_j)$,
a homomorphism of group schemes $C_j \to \R_{T'/T} A$ defined by
taking $\phi$ to $(\delta \phi)(g_0, \ldots, g_{j+1})$. Let $Z_j$ be
the intersection of the kernels of all these morphisms. Denote the
corresponding representable functors by $\ms C^j, \ms Z^j$. We then
have, by definition of cohomology, a short exact sequence of
functors:
\[\ms C^{i-1} \to \ms Z^i \to \, ^G\ms H^i_{T'/T, A} \to 0,\]
showing that $^G \ms H^i_{T'/T, A}$ is presentable as claimed.
\end{proof}

\begin{defn}
For a commutative algebraic group $A$ (or more generally, Abelian
functor) over $T$, 
let $\ms H^i_{A/T} = \ms H^i_A \in \abfun T$ denote the functor defined by
$\ms H^i_A(R) = H^i_\et \left(R, A_R\right)$.
\end{defn}

\begin{defn}
Let $T$ be a commutative ring and $R$ a commutative $T$-algebra. Suppose we are
given a finite group $G$, and a $R$-algebra $R'$ with $G$-action such that
$R'$ is a $G$-Galois extension of $R$. We say that $R'/R$ is
\textit{pointwise $G$-Galois versal} if for every $T$-algebra $L$ which is
a field, and for every $G$-Galois extension $E/L$, there exists a
homomorphism of $T$-algebras $R \to L$ such that $E \cong R' \otimes_R L$
as $G$-algebras.
\end{defn}

\begin{lem} \label{versal exist}
Let $T$ be a commutative ring. Then there exist
pointwise $G$-Galois versal extensions of commutative $T$-algebras $R'/R$
for every finite group $G$.
\end{lem}
\begin{proof}
Let $M$ be the free $\mbb Z$-module with generators $x_g$ in bijection with
the elements of $G$, let $R' = T[M][s^{-1}]$ be the localization of the
group algebra of $M$, where $s = \prod_{g \neq h} (x_g - x_h)$, endowed by
the natural action of $G$, and $R = (R')^G$. This is a $G$-Galois
extension, by \cite[Ch.III,~Prop.~1.2(5)]{DeIn}.

To see that $R'/R$ is in fact pointwise $G$-Galois versal, let $E/L$ be a
$G$-Galois extension of $T$-algebras with $L$ a field. By the normal basis
theorem \cite[Theorem~4.30]{Jac:BAI}, we may find a basis for $E/L$
consisting of distinct elements $y_g$, permuted by the Galois group in the
obvious way.  Consequently, by sending $x_g$ to $y_g$, we obtain a
homomorphism $R' \to E$ preserving the Galois action, and hence a
homomorphism $R \to L$. It follows immediately that $L \otimes_R R' \cong
E$ as desired.
\end{proof}

\begin{rem} \label{Galois Noetherian}
It follows from the construction above that if $T$ is Noetherian, so is
$R$. To see this, we note first that $R'$ is Noetherian by the Hilbert
Basis Theorem, since it is finitely generated over $T$. On the other hand,
since $R'/R$ is Galois, it follows from \cite[Ch.III,~Prop.~1.2(3)]{DeIn}
that $R'$ is finitely generated over $R$. By the Eakin-Nagata theorem (see,
for example \cite[Theorem~3.7]{Mat}), we may then conclude that $R$ is
Noetherian.
\end{rem}

Let $T$ be a commutative ring, $G$ a finite group, and $A \in \abfun T$.
For $S'/S$ a $G$-Galois extension of commutative $T$-algebras, the
Hochschild-Serre spectral sequence (\cite[Theorem~2.20]{Milne:EC}):
\[H^p(G, H^q(S', A)) \Longrightarrow H^{p+q}(S, A)\]
yields an ``edge morphism'' $H^i(G, A) \to H^i(S, A)$, which in the case
that $S'/S$ is an extension of fields, coincides with the standard
inflation map in Galois cohomology. Define $^G \ms H^i_A \in \fun T$ to be
the subfunctor of $\ms H^i_{A/T}$ whose $S$-points consists of those
classes in $H^i(S, A)$ which may be represented as the image of a cocycle
in $H^i(\Gal(S'/S), A)$ for some Galois extension $S'/S$ with group $G$. We
note that although $\ms H^i_{A/T}$ is an Abelian functor, $^G \ms H^i_A$ is
valued in sets and not Abelian groups.

\begin{prop} \label{G coverability}
Let $T$ be a commutative ring, $G$ a finite group, and $A \in \abfun
T$. Then $^G \ms H^i_{A}$ is pointwise coverable.
\end{prop}
\begin{proof}
Via Lemma~\ref{versal exist}, we may choose a pointwise $G$-Galois
versal extension of commutative $T$-algebras $R'/R$. 
Suppose $H$ is an affine scheme which is a pointwise cover
of $^G \ms H^i_{R'/R, A}$ as in Lemma~\ref{presentable cohomology}.
$H$ is a $R$-scheme, but may also be considered as an $T$
scheme by virtue of the $T$-algebra structure of $R$. We obtain a
morphism:
\[\phi: H \to \, ^G \ms H^i_A\]
as follows. If $S$ is a commutative $T$-algebra, a element of $H(S)$ yields a
morphisms $\Spec(S) \to
^G \ms H^i_{R'/R, A}$
which is equivalent to giving a $R$-algebra structure to $S$ (via the
structure morphism $^G \ms H^i_{R'/R, A} \to \Spec(R)$) together with an
class in $H^i(G, A(R' \otimes_R S))$. Since
$R' \otimes_R S/S$ is a $G$-Galois extension of rings, and in
particular an \'etale $G$-cover of the corresponding schemes, we
obtain a natural morphism $H^i(G, A(R' \otimes_R S)) \to H^i(S, A)$ as
the edge maps of the Hochschild-Serre spectral sequence
$H^p(G, H^q(R' \otimes_R S, A)) \Longrightarrow H^{p+q}(S, A)$
giving finally an element of $H^i(S, A)$
determined by our element of $H(S)$.

To see that $H$ is a pointwise cover of $^G \ms H^i_A$, suppose that
$E/L$ is any $G$-Galois extension of fields with $T$-algebra
structure, and that $\alpha \in H^i(L, A)$ is the image of $\beta \in
H^i(E/L, A(E))$. By definition of $R'/R$, we may find a homomorphism
of rings $\psi : R \to L$ such that $E/L$ is the pullback of $R'/R$
along $\psi$ --- i.e.  we have an isomorphism of $G$-algebras $R'
\otimes_{R, \psi} L \cong E$. In particular, $\psi$ gives $L$ the
structure of a $R$-algebra, and by definition, $\beta$ corresponds to
an $L/R$-point of $^G \ms H^i_{R'/R, A}$. But, therefore, by
definition of $H$ this comes from a point of $H(L)$ which in turn maps
to $\alpha$, as desired.
\end{proof}

\subsection{Bounding symbol length in terms of the representation index}

\begin{defn}
Let $\ms K_M^i$ be the Abelian functor $\ms K_M^i(R)= K_M^i(R)$, where
$K_M^\bullet(R) = T_{\mbb Z}^\bullet(R^*)/\left<a \otimes (1 - a) \mid a, (1 -
a) \in R^*\right>$ is the ``naive'' Milnor $K$-theory of $R$ (as in
\cite{Kerz}), and let $\ms K_\ell^i$ be defined by $\ms K_\ell^i(R) =
K^i_M(R)/\ell K^i_M(R)$. We let $^m \ms K_M^i$ and $^m \ms K_\ell^i$ denote
the subfunctors of $\ms K_M^i$ and $\ms K_\ell^i$ respectively consisting
of those classes which may be expressed as a sum of at most $m$ symbols.
\end{defn}

\begin{rem} \label{K coverable}
It follows immediately from the definition that each of the functors
$^m \ms K^i_M$ is coverable by a finite number of copies of $\mbb
G_m$.
\end{rem}

\begin{rem}
We have a morphism of Abelian functors $\mbb G_m = \ms K_\ell^1 \to \ms
H^1_{\mu_\ell}$ induced by the boundary map in the long exact sequence of
\'etale cohomology induced by the short exact sequence
\[1 \to \mu_\ell \to \mbb G_m \to \mbb G_m \to 1\] 
Using the cup product in cohomology, this induces morphisms:
$\ms K_\ell^i \to \ms H^i_{\mu_\ell^{\otimes i}}$ for all $i$.
\end{rem}

\begin{defn} \label{rep index defn}
Let $\alpha \in H^d(F, \mu_\ell^{\otimes d})$, we define the
\textit{representation index} of $\alpha$, denoted $\rind(\alpha)$ to be
the minimum degree of a finite Galois extension $E/F$ such that $\alpha$ is
in the image of the inflation map $H^d(\Gal(E/F), \mu_\ell^{\otimes d}(E))
\to H^d(F, \mu_\ell^{\otimes d})$.
\end{defn}

\begin{thm} \label{rep to length bound}
Let $\alpha \in H^d(F, \mu_\ell^{\otimes d})$. Then we may bound the
length of $\alpha$ only as a function of the representation index of
$\alpha$.
\end{thm}
\begin{proof}
Fix a finite group $G$, and let $T'/T$ be a pointwise $G$-Galois
versal extension of $\mbb Z$-algebras. By Proposition~\ref{G
coverability}, the functor $^G \ms H^d_{\mu_\ell^{\otimes d}}$ is
a pointwise coverable subfunctor of $\ms H^d_{\mu_\ell^{\otimes d}}$.
It follows from the Bloch-Kato conjecture that this subfunctor is
pointwise contained in the image of the union $\bigcup_{m > 0}  {}^m\ms
K^d_\ell$. By Remark~\ref{K coverable}, these are coverable, and so
by Theorem~\ref{functor compactness} (using the fact that $T$ is
Noetherian as in Remark~\ref{Galois Noetherian}), it follows that
$^G \ms H^d_{\mu_\ell^{\otimes d}}$ is contained in a finite union
$\bigcup_{0 < m < M} {}^m\ms K^d_\ell$. But this precisely says that
for any cohomology class $\alpha \in H^d(F, \mu_\ell^{\otimes d})$
for any field $F$ of characteristic not divisible by $\ell$, we may
write $\alpha$ as a sum of no more than $M$ symbols. In particular,
$M$ depends only on the group $G$.

In particular, given any bound on the representation index of
$\alpha$, we may come up with an exhaustive finite list of groups $G$
such that $\alpha$ arises as the inflation of a class $H^d(\Gal(E/F),
\mu_\ell^d)$ for a Galois extension $E/F$ with group $G$.  For each
such $G$, since the previous paragraph gives the existence of a bound
$M_G$ for the symbol length for such classes, we may write any such
class as a sum of at most $M = \max \{M_G\}$ symbols.
\end{proof}

\section{Bounding representation index in terms of effective
indices}

\begin{thm} \label{bounded rep thm}
Pick a positive integer $n$, and suppose that for every $m > 0$ and $q \leq n-1$
there is an integer $N_{q,m}$ such that for every finite field extension
$L/F$ with $[L:F] \leq m$, and $\alpha \in H^q(L, \mu_\ell)$, $\eind \alpha
\leq N_{q,m}$.  Then every class in $H^n(F, \mu_\ell)$ has representation index
bounded in terms of its effective index.
\end{thm}
We note here that since we are not focusing on constructing an optimal
bound, we are not concerned with the twist on the roots of unity;
after a extension of degree at most $\Phi(\ell)$ we may assume
$\mu_\ell \subset F$ in any case.

Before proceeding to the proof of this theorem, we first note that it
quickly implies the following:

\begin{thm}\label{main bound}
Let $F$ be a field and fix a positive integer $\ell$ prime to the
characteristic of $F$. Fix $d > 0$. Suppose that for every $m > 0$, there
is some $N_m$ such that for every cohomology class $\beta \in H^{d'}(L,
\mu_\ell^{\otimes d'})$ with $[L:F] \leq m$ and $d' < d$ we have $\eind
\beta \leq N_m$.  Then for $\alpha \in H^d(F, \mu_\ell^{\otimes d})$, we
may bound $\lambda(\alpha)$ in terms of the effective index of $\alpha$.
\end{thm}
\begin{proof}
Under these hypotheses, by Theorem~\ref{bounded rep thm}, we may bound
the representation index. But by Theorem~\ref{rep to length bound}, we
may therefore bound the length of $\alpha$.
\end{proof}

To prove Theorem~\ref{bounded rep thm}, we will need to make
detailed use of the Hochschild-Serre spectral sequence, which we now
recall:

Suppose that $K/F$ is a $G$-Galois extension, $\mc G = \Gal(F)$, $\mc
N = \Gal(K)$ (so that $\mc G/\mc N = G$). Then we have a convergent
spectral sequence
\[H^p(G, H^q(\mc N, \mu_\ell)) = E^{p,q}_2 \Longrightarrow E^{p+q} =
H^{p+q}(\mc G, \mu_\ell)\]
This comes with a filtration on $E^n$ which we denote by $\mc
F^\bullet E^n$, with $\mc F^{i+1} E^n \subset \mc F^{i} E^n$, and
$F^i E^n/F^{i+1} E^n \cong E^{i, n-i}_\infty$.

We will have to make use of this spectral sequence for different
choices of subgroups $\mc N \subnorm \mc G$. In this case, we will
write the filtrations as $\mc F_{\mc N \subnorm \mc G}^\bullet$, and
the spectral sequence terms as $E^{p,q}_{\mc N \subnorm G,j}$ and
$E^n_{\mc N \subnorm G}$. Note that if we have normal subgroups $\mc
N' \subset \mc N$ of $\mc G$, we obtain a morphism of spectral
sequences
\[\xymatrix{
E^{p,q}_{\mc N, \subnorm G} \ar@{=}[r] \ar[d] & H^p(\mc G/\mc N,
H^q(\mc N, \mu_\ell)) \ar[d] \ar@{=>}[r] & H^{p+q}(\mc G, \mu_\ell)
 \ar@{=}[d] & \mc F^\bullet_{\mc N \subnorm \mc G}E^{p + q}_{\mc N \subnorm \mc G} \ar[d] \\
E^{p,q}_{\mc N', \subnorm G} \ar@{=}[r] & H^p(\mc G/\mc N',
H^q(\mc N', \mu_\ell)) \ar@{=>}[r] & H^{p+q}(\mc G, \mu_\ell)
 & \mc F^\bullet_{\mc N' \subnorm \mc G} E^{p + q}_{\mc N' \subnorm \mc G} \\
}\]
that is, a map on each page which commutes with the differentials, and
induces a corresponding map on the filtered parts. In this case, the
maps on the $E_2$-page are induced by a combination of inflation and
restriction: the inclusion $\mc N' \subset \mc N$ gives a restriction
map $H^q(\mc N, \mu_\ell) \to H^q(\mc N', \mu_\ell)$ which then gives
\[\xymatrix{
Map((\mc G/\mc N)^{p+1}, H^q(\mc N, \mu_\ell)) \ar[r] \ar@{-->}[rd]& Map((\mc G/\mc
N)^{p+1}, H^q(\mc N', \mu_\ell)) \ar[d] & \\ 
& Map((\mc G/\mc N')^{p+1}, H^q(\mc N', \mu_\ell)),
}\]
where the latter map is obtained by pre-composition. Together the
resulting dashed arrow above give homomorphisms
\[H^p(\mc G/\mc N, H^q(\mc N, \mu_\ell)) \to H^p(\mc G/\mc N', H^q(\mc
N', \mu_\ell)),\] which gives the map on the $E_2$ page.

\begin{proof}[proof of Theorem~\ref{bounded rep thm}]
Let $\alpha \in \ker H^n(F, \mu_\ell) \to H^n(K, \mu_\ell)$. Without
loss of generality, we may assume that $K/F$ is Galois, say with
Galois group $G$. Let $\til K/K$ be some field extension with $\til
K/F$ Galois with group $G' = \mc G/\mc N'$, and consider
the Hochschild-Serre spectral sequence for $\mc N' \subnorm \mc G$. 
By definition, we know that our element $\alpha$ is in the filtered
part $\mc F^{1}_{\mc N' \subnorm \mc G} H^n(F, \mu_\ell)$ (since the restriction map to
$H^n(\til K, \mu_\ell)^G$ coincides with the map $E^n_{\mc N' \subnorm
G} \to E^{0, n}_{\mc N' \subnorm G, 2}$).
Since we have a surjection 
\[E^{n, 0}_{\mc N' \subnorm G, 2} = H^n(G', \mu_\ell) \to \mc F^n_{\mc
N' \subnorm \mc G} H^n(F, \mu_\ell) = \mc
F^n_{\mc N' \subnorm G} E^n_{\mc N' \subnorm G},\] 
it follows that $\alpha$ is in the image of a class $\alpha'
\in H^n(G', \mu_\ell)$ exactly when it lies in the $n$'th filtered
part of $E^n_{\mc N' \subnorm G}$.

To proceed, we will inductively construct our field extension $\til K$
whose degree is bounded in terms of $[K:F]$ and $n$ such that the
class $\alpha$ is in the correct filtered part with respect to the
corresponding Hochschild-Serre spectral sequence.

For the sake of notation, let us begin with the case $\til K = K$ and
let us assume that we have $\alpha \in \mc F^i_{\mc N \subnorm \mc G}
H^n(F, \mu_\ell)$, with $0 < i < n$. We will show that we may find
$\til K$ with $\alpha \in \mc F^{i+1}_{\mc N' \subnorm \mc G} H^n(F,
\mu_\ell)$ where the degree of $\til K/K$ is bounded in terms of $n$
and $[K:F]$. Since the number of such $i$'s we must consider is at
most $n-1$, this will prove the claim. 

Consider the image $\ov \alpha$ of $\alpha$ in the subquotient
$F^i_{\mc N \subnorm \mc G}H^n(F, \mu_\ell)/F^{i+1}_{\mc N \subnorm
\mc G}H^n(F, \mu_\ell) =
E^{i,n-i}_\infty$. This group is the image of a subgroup 
\[Z^{i, n-i} \subset E^{i, n-i}_2 = H^i(G, H^{n-i}(\mc N, \mu_\ell))\]
consisting of those elements such that all the successive
differentials vanish. In particular, we may choose a representative
$\til \alpha \in Z^{i, n-i} \subset H^i(G, H^{n-i}(\mc N, \mu_\ell))$
for $\ov \alpha$, and a representative cochain for $\til \alpha$
represented as a function from $G^{i+1}$ to $H^{n-i}(\mc N, \mu_\ell) =
H^{n-i}(K, \mu_\ell)$. By hypothesis, the effective indices of the images
of this function are bounded by $N_{n-i, [K:F]}$ since $n-i < n$. Since we
need only consider $|G|^{i+1}$ such classes, we may find an extension $\til
K/K$ of degree at most $N_{n-i, [K:F]}^{|G|^{i+1}}$ which splits these
classes. In particular, replacing $\til K$ with its Galois closure over
$F$, also of size bounded inductively in terms of the previous extension
degree $[K:F]$, we have found an extension $\til K/K$ of bounded size with
the image of $\alpha$ in the filtered part $F^{i+1}_{\mc N' \subnorm \mc G}
H^n(F, \mu_\ell)$ as desired.
\end{proof}

\section{Cohomology at $2$, Pfister numbers and index bounds} \label{coh2}

For Galois cohomology with coefficients in $\mu_2$, we may use the
tools of quadratic form theory to obtain, in certain cases, bounds on
the indices of cohomology classes, in particular, in terms of the
$u$-invariant. These results are fairly direct applications of the
Milnor conjectures and other results in the theory of quadratic forms. 

We will have use of the following theorem proved by Karpenko,
originally conjectured by Vishik:
\begin{thm}[\cite{Kar:holes}, conjecture 1.1] \label{hole}
If $\phi$ is an anisotropic quadratic form such that $\phi \in I^n
\subset W(F)$, where $I$ is the fundamental ideal of the Witt ring,
and $dim(\phi) < 2^{n+1}$, then $dim(\phi) = 2^{n+1} - 2^i$ for $1
\leq i \leq n+1$.
\end{thm}
Note that by convention, we consider the $0$ form to be anisotropic.

\begin{prop} \label{main2}
Suppose $F$ is a field of characteristic not equal to $2$ and with $u$
invariant $u(F) \leq 2^n$ (for example, a $C_n$ field), and $\alpha \in
H^m(F, \mu_2)$. Then
\begin{enumerate}
\item \label{cd}
If $m > n$ then $\alpha = 0$. In other words, $\cd_2(F) \leq n$.
\item \label{p_eq_i}
If $m = n$ then $\eind(\alpha) | 2$.
\item \label{pi_bound}
If $m < n$, then $\eind(\alpha) | 2^{2^{n-1} - 2^m + m + 1}$.
\end{enumerate}
\end{prop}
\begin{proof}
In each case, by the Milnor conjecture (\cite{OVV} for characteristic $0$,
or Merkurjev's article in \cite{Merk:MC} for the case of general characteristic not $2$), we may
choose a quadratic form $\phi \in I^m$ such that $e_m(\phi) =
\alpha$. By our assumption on the $u(F)$, we may represent $\phi$ by
an anisotropic form of dimension at most $2^n$. By theorem
\ref{hole}, if $dim(\phi) < 2^{m+1}$ it must be of the form $2^{m+1}
- 2^i$ for some $1 \leq i \leq m+1$.

In case \ref{cd}, we know that $\phi$ may be represented by an
anisotropic form of dimension less than $2^n$, which is in turn less
than $2^m$. However, $2^{m+1} - 2^i$ is either $0$ or greater than or
equal to $2^n$. Therefore $\phi = 0$ and $\alpha = e_m(\phi) = 0$ as
well.

Alternately, since $u(F) \leq 2^n$, every $m$-fold Pfister form is
trivial. Since $I^m$ is generated by $m$-fold Pfister forms, it must
be trivial. Therefore, $\phi = 0$ and $\alpha = 0$.

In case \ref{p_eq_i}, the dimension of an anisotropic representation
of $\phi$ has dimension at most $2^n$, and must be of the form
$2^{n+1} - 2^i$. Therefore it is either $0$ or of dimension exactly
$2^{n+1} - 2^n = 2^n$. If we then choose a quadratic extension $L/F$
such that $\phi_L$ is isotropic, then its anisotropic part will be a
strictly smaller dimensional form in $I^n$, which again by
theorem \ref{hole} immediately forces it to be $0$. Therefore
$\eind(\alpha) | 2$.

In case \ref{pi_bound}, note that we may always find a tower of
quadratic extensions such that each successive extension lowers the
dimension of the anisotropic part of $\phi$ by at least $2$. Since the
anisotropic part has dimension at most $2^n$, we may find a field extension
$L/F$ which is a tower of 
$$\frac{2^n - (2^{m+1} - 2)}{2} = 2^{n-1} - 2^m + 1$$ many quadratic
extensions such that $\phi_L$ has anisotropic part at most $2^{m+1} -
2$. By lemma \ref{hole_jump}, $\phi_L$ is split by an extension of
degree dividing $2^m$. Therefore, $\phi$ is split by an extension of
degree
$$2^{m}2^{2^{n-1} - (2^{m} - 1)} = 2^{2^{n-1} - 2^m + m + 1}.$$
\end{proof}

\begin{lem} \label{hole_jump}
Suppose $\psi \in I^m$ has dimension at most $2^{m+1} - 2^i$. Then
there is an extension $K/F$ of degree dividing $2^{m + 1 - i}$ which
splits $\psi$.
\end{lem}
\begin{proof}
Choose any quadratic extension $E/F$ such that $\psi_E$ is
isotropic. By theorem \ref{hole}, its anisotropic part must have
dimension at most $2^{m+1} - 2^{i+1}$ and so by induction it is split
by an extension $K/E$ of degree dividing $2^{m + 1 - (i + 1)} = 2^{m -
i}$. But this implies that $[K : F]$ divides $2^{m - i + 1}$ and we
are done.
\end{proof}

It seems extremely unlikely that the bounds obtained above on the
effective index are the optimal ones. Rather, it seems more reasonable
to expect the more optimistic conjecture, which came out of
discussions at the AIM conference ``Deformation theory, patching,
quadratic forms, and the Brauer group'':
\begin{conj}
Suppose $\alpha \in H^n(F, \mu_\ell^{\otimes n})$, and the Diophantine
dimension of $F$ is bounded by $d$. Then 
\[\eind(\alpha) | \ell^{\binom{d-1}{n-1}}.\]
\end{conj}
Let us describe briefly some circumstantial evidence for this.  The main
support for this conjecture comes from the case $\ell = 2$.  If we suppose
that $F$ with Diophantine dimension bounded by $d$, then it follows that
$u(F) = 2^d$. An anisotropic form would then be at most dimension $2^d$,
and in this case, it would be Witt equivalent to an orthogonal sum $q \cong
b_1 \perp \cdots \perp b_{2^{n-1}}$ for binary forms $b_i$. At worst, this
could be split by individually splitting each of the binary forms in
successive quadratic extensions, yielding a splitting field of degree
$2^{2^{d-1}}$. 
On the other hand, another way to split the form would be to split
each of its cohomological invariants, starting from $e_1$ all the way
to $e_d$. If we assume that $2^{n_i}$ is the effective index of the
$e_i$ invariant obtained at the $i$'th stage, then we should obtain a
splitting field of degree $2^{\sum_{i = 1}^{d}n_i}$. Assuming that the
$e_i$ invariants are ``independent,'' and ``general,'' we expect then
that $2^{n_i}$ should represent the maximal effective index of a class
in $H^i(F, \mu_2)$ (implicitly assuming that this bound should stay
the same over finite extensions of $F$).
We therefore expect that if $2^{n_i}$ is the maximal effective index
of a class in $H^i(F, \mu_2)$, then $\prod_{i = 1}^d 2^{n_i} =
2^{2^{d-1}}$. In other words
\[n_1 + n_2 + \cdots + n_d = 2^{d - 1}\]
Further, we have $n_1 = 1 = n_d$ from Theorem~\ref{main2}, and
conjecturally $n_2 = d - 1$ (see \cite[page~2]{Lie:PIS}, where it is
attributed to a question of Colliot-Th\'el\`ene). Besides this, the
conjecture that $n_i$ should at worst be $\binom{d-1}{i-1}$ is simply
based on the above numerology and the feeling that in some sense the
cohomology ring ``feels a lot like an exterior algebra.'' 

\smallskip

Recall that the $m$-Pfister number of a quadratic form $q$,
denoted $\Pf_m(q)$ is the smallest number $j$ such that $q$ is Witt
equivalent to a sum of $\phi_1 \perp \phi_2 \perp \cdots \perp \phi_j$
of $m$-fold Pfister forms. This terminology is due to Brosnan,
Reichstein and Vistoli \cite[Section~4]{BRV}.

\begin{thm} \label{pfister}
Suppose that $F$ is a field of characteristic not $2$ in which $-1 \in
(F^*)^2$. Then the following are equivalent:
\begin{enumerate}
\item \label{u case} $u(F) < \infty$.
\item \label{strong ind case} For all $m > 0$, there exists $N_m > 0$ such that
for all $n > 0$ and for all $\alpha \in H^n(L, \mu_2)$ with $[L:F] \leq m$,
we have $\eind \alpha < N_m$.
\item \label{len case} There exists $N > 0$ such that for all $n > 0$
and for all $\alpha \in H^n(F, \mu_2)$, we have $\lambda(\alpha) < N$.
\item \label{pf case} There exists $N > 0$ such that for all $n > 0$
and for all $q \in I^n(F)$, we have $\Pf_n(q) < N$.
\end{enumerate}
\end{thm}
\begin{proof}
\mbox{\ }

\noindent
(\ref{u case})
$ \ \Longrightarrow \  $ 
(\ref{strong ind case}): 
By \cite{Leep:SQF}, $u(F) < \infty$ implies that for
every finite extension $L/F$, $u(L)$ is bounded in terms of the
extension degree. In particular, without loss of generality we may assume
$L = F$, and so we must show that if $u(F) < \infty$ then $\eind(\alpha)$
is universally bounded for $\alpha \in H^n(F, \mu_2)$ for all $n$.  But
this follows directly from Proposition~\ref{main2}.

\noindent
(\ref{strong ind case})
$ \ \Longrightarrow \  $ 
(\ref{len case}): 
Our hypothesis implies that, in particular, $\eind \alpha < N_1$ for
$\alpha \in H^2(F, \mu_2)$. 
It follows from the remark just after Proposition~\ref{cdim bound}, that
we may bound the $2$-cohomological dimension of $F$. But then,
Theorem~\ref{main bound} applies in each of the finite number of
relevant degrees to give a bound on the length of any mod-$2$
cohomology class.

\noindent
(\ref{len case})
$ \ \Longrightarrow \  $ 
(\ref{pf case}): 
It follows from Proposition~\ref{cdim bound} that we may bound the
$2$-cohomological dimension of $F$. Say $H^m(F, \mu_2) = 0$ for some
$m > 0$. Let 
\[\Pf_{n}(F) = \max_{q \in I^n(F)}\{\Pf_n(q)\}.\]
We proceed to show that $\Pf_{n} < \infty$ for all $n$ by descending
induction on $n$. To begin, we know that since $H^m(F, \mu_2) = 0$, it
follows from the Milnor Conjectures (or the mod-$2$ case of the
Bloch-Kato Conjecture), proved in \cite{Voe:RPO,Voe:MC2} (see also
Merkurjev's article in \cite{Merk:MC} for general characteristic not $2$), or more generally the
Bloch-Kato conjecture/Norm residue isomorphism theorem \cite{Voe:modl,Wei:NRI}, that $H^{m'}(F,
\mu_2) = 0$ and therefore, that $I^{m'}(F) = 0$ for all $m' \geq m$
since the canonical map from Milnor K-theory to the graded Witt ring
is surjective (see \cite{Milnor:AKTQF}). Now, suppose
that we have shown $\Pf_{n} < \infty$. Consider $q \in I^{n-1}(F)$.
Since $\lambda(e_{n-1}(q)) < N$, we may write $q$ as $q = q' + r$
where $r$ is a sum of at most $N$ $(n-1)$-fold Pfister forms and where
$q' \in I^n(q)$. By induction, we know that $\Pf_{n}(q') < \Pf_n(F) <
\infty$, and so we may write $r = \sum_{i = 1}^{\Pf_n(F)} r_i$. But if
$n > 1$, each $r_i$ is the sum of three $(n-1)$-fold Pfister forms: if
we write $r_i = <1, -a><1, -b>\phi$, where $\phi$ is an $(n-2)$-fold
Pfister form (or $1$ if $n = 2$), then since we have (using the fact that $<1, 1> = <1,
-1>$, since $-1 \in (F^*)^2$):
\[<1, -a><1, -b> = <1, -a> \perp <-b, ab> \sim <1, -a> \perp <1, -b>
\perp <1, ab>\]
and so $r_i = <1, -a>\phi \perp <1, -b>\phi \perp <1, ab> \phi$.
Therefore, we have $\Pf_{n}(F) < N + 3\Pf_{n-1}(F) < \infty$, as
desired. 

In the case $n = 1$, it is immediate that every form in
$I(F)$ is the sum of a binary form and one in $I^2(F)$. Using the
argument of the preceding paragraph, one has
$\Pf_1(F) \leq 1 + 3\Pf_2$. 

\noindent
(\ref{pf case})
$ \ \Longrightarrow \  $ 
(\ref{u case}): 
This follows immediately upon considering the case $n = 1$, since
every form differs from a form in $I(F)$ by at most a $1$-dimensional
form.
\end{proof}

\section{Generic splitting varieties in degree $2$}

Suppose $\ms H$ is a functor from the category of $F$-algebras to the
category of Abelian groups, and let $\alpha \in \ms H(F)$.

\begin{defn}
We say that a variety (resp. scheme) $X/F$ is a generic splitting
variety (resp. scheme) for $\alpha$ if for every field extension
$L/F$, we have $X(L) \neq \emptyset$ if and only if $\alpha_L = 0$.
\end{defn}

We follow \cite[0.1]{CTSan} and use the term multiplicative type to
mean a finite type scheme over $F$ which is diagonalizable after
extending scalars to a suitable separable extension $L/F$ (a finite
type isotrivial group scheme in the language of \cite[X]{SGA3-8to11}).
In particular, this includes the cases of smooth finite group schemes
as well as tori.

In the case that a functor is (almost) presentable, we may in fact
construct such generic splitting schemes easily, and this will be the
basis of our construction:

\begin{prop} \label{presentable splitting}
Suppose $\ms H' \in \abfun F$ is presentable, and let $\ms H$ be any
Abelian functor together with a map $\phi: \ms H' \to \ms H$ which is
an injection on $L$-points for every $L/F$ a field extension. Let
$\alpha \in \ms H(F)$ be the image of $\alpha' \in \ms H'(F)$ under
$\phi$. Let
\[\xymatrix{
A_1 \ar[r]^g & A_0 \ar[r]^f & \ms H' \ar[r] &  0
}\]
be a presentation and $\til \alpha \in A_0(F)$ be a preimage of
$\alpha'$. Then $X = g^{-1}(\til \alpha)$ is a generic splitting scheme
for $\alpha'$.
\end{prop}
\begin{proof}
This follows essentially immediately from the definitions of
presentability and the injectivity of $\phi$. Since the map $f$ is
surjective as a map of functors, it follows that the preimage exists,
and further, for a given field
extension $L/F$, we will have $\alpha_L' = f(\til \alpha_L) = 0$
exactly when $\til \alpha_L$ is the image of a class in $A_1(L)$,
which is to say, exactly when $X(L)$ is nonempty. Since $\alpha_L' =
0$ if and only if $\alpha_L = 0$ by injectivity of $\phi$, we are
done.
\end{proof}

\begin{thm} \label{generic splittings degree 2}
Let $A$ be any commutative group scheme over $F$ of multiplicative
type and suppose $\alpha \in H^2(F, A)$.  Then there exists a smooth
generic splitting variety $X/F$ for $\alpha$.
\end{thm}

\begin{proof}
By the definition of multiplicative type and the Weil restriction, we
may find an separable extension $L/F$ and an embedding of $A$ into
$\R_{L/F} \mbb G_m^n$.

Now, choose a field extension $E/F$ such that $\alpha_E = 0$, and
consider the composition of embeddings 
\[A \to \R_{L/F} \mbb G_m^n \to \mc T = \R_{E/F} \left(\left(\R_{L/F} \mbb
G_m^n\right)_E\right).\]
We then find that, using
Shapiro's lemma as in \cite[VII(29.6)]{BofInv}, that the image of
$\alpha$ in $H^2(F, \mc T) = H^2\left(E, \left(\R_{L/F} \mbb
G_m^n\right)_E\right)$ is trivial. Further we have, again by Shapiro's
lemma, also in the context of \'etale algebras (see
\cite[VII(29.7)]{BofInv}) 
\[H^1(F, \mc T) = H^1\left(E, \left(\R_{L/F} \mbb G_m^n\right)_E\right) =
H^1\left(E, \R_{L \otimes E/E} \mbb G_m^n\right) = H^1(L \otimes
E, \mbb G_m^n) = 0\]
using Hilbert's theorem 90.
In particular, if we let $\mc T' = \mc T/A$ be the quotient torus, then using
the exact sequence in Galois cohomology, we obtain:
\[0 \to H^1(F, \mc T') \to H^2(F, A) \to H^2(F, \mc T)\]
Which implies, in particular, that the image of $H^1(F, \mc T')$ in
$H^2(F, A)$ contains the element $\alpha$.

Using \cite[Proposition~1.3(1.3.1)]{CTSan}, we may find an exact
sequence of tori 
\[0 \to \mc T' \to \mc P \to \mc Q \to 0\] 
where $\mc P = \R_{K/F} \mbb G_m$ for $K/F$ an \'etale algebra is a
quasitrivial torus and $\mc Q$ is coflasque. In particular, we obtain
for every $F$-algebra $T$, an exact sequence:
\[\mc P(T) \overset{f}{\to} \mc Q(T) \to H^1(T, \mc T') \to H^1(T, \mc P)\]
in particular, if we let $\ms H'$ be the quotient cokernel functor
(not sheaf!) of $\mc P \to \mc Q$, and $\ms H \in \abfun F$ be defined
by $\ms H(T) = H^2(T, A)$, then we obtain a morphism $\ms H' \to \ms
H$ via the composition
\[\mc Q(T) \to H^1(T, \mc T') \to H^2(T, A)\]
which, using the fact that $\mc P$ is quasitrivial, induces a
injection $\ms H'(L) \to \ms H(L)$ for every field extension $L/F$.
Since $\alpha$ is in the image of this map on $F$-points, say of a
point $\alpha' \in \ms H'(F)$, it follows that $\phi^{-1}(\til
\alpha)$ is a generic splitting scheme for $\alpha$, where $\til
\alpha$ is a preimage of $\alpha'$ in $\mc Q(F)$. But from the
definition of $\ms H'$, it follows from Proposition~\ref{presentable
splitting} that $X = \phi^{-1}(\til \alpha)$ is a generic splitting scheme
for $\alpha$. But since $X$ is a homogeneous variety under the action
of the torus $\ms P$, it follows that it is a smooth variety as
desired. 
\end{proof}

\bibliographystyle{alpha}
\bibliography{citations}

\end{document}